\documentclass[12pt]{article}
\usepackage{amsmath, amscd, amsfonts, amssymb, amsthm, tikz, times, setspace, srcltx, verbatim}
\newtheorem{theorem}{Theorem}[section]
\newtheorem{lemma}{Lemma}
{\bf}{\rm}
    \numberwithin{remarks}{section}
    \newtheorem{proposition}[theorem]{Proposition}
\numberwithin{theorem}{section} \numberwithin{definition}{section}

\newcommand{\eqnsection}{
    \renewcommand{\theequation}{\thesection.\arabic{equation}}
    \makeatletter
    \csname @addtoreset\endcsname{equation}{section}
    \makeatother}

\newcommand{\dd}{\delta}

\newcommand{\lar}{\longrightarrow}
\newcommand{\eps}{\varepsilon}

\newcommand{\aaa}{\alpha}
\newcommand{\reals}{\mathbb{R}}

\newcommand{\unn}{\bigcup}

\def \be{\begin{equation}}
\def \ee{\end{equation}}
\def \bt{\begin{theorem}}
\def \et{\end{theorem}}
\def \bl{\begin{lemma}}
\def \el{\end{lemma}}
\def \bea{\begin{eqnarray}}
\def \eea{\end{eqnarray}}
\def \bas{\begin{eqnarray*}}
\def \eas{\end{eqnarray*}}

\def \lll{\label}
\newcommand {\rrr}[1]{(\ref{#1})}

\def \aa{\alpha}
\def \bb{\beta}
\def \ga{\gamma}

\def \la{\lambda}

\def \ka{\kappa}

\def \ff{\infty}

\def \R{\mathbb{R}}

\def \AA{{\cal A}}
\def \BB{{\cal B}}

\def \DD{{\cal D}}

\def \FF{{\cal F}}

\def \II{{\cal I}}

\def \RR{{\mathbb R}}

\def \TT{{\cal T}}

\def \E{\mathbf{E}}
\def \P{\mathbf{P}}

\def \({\left(}
\def \){\right)}

\def \nn{\nonumber}

\def \vski{\vspace{12pt}}
\def \sgn{\mathop{\mathrm{sgn}}}

\def \bc{\begin{center} }
\def \ec{\end{center} }
\def \bs{\begin{slide} }
\def \es{\end{slide} }
\def \noi{\noindent}

\def\dslt{\hat \alpha_{t}'}
\def\dsltA{\hat \alpha_{t}'(y,A)}
\def\dslteps{\hat \alpha_{t,\eps}'}
\def\dsltepst{\hat \alpha_{\tilde{t},\eps}'}
\def\dsltepse{\hat \alpha_{t,\tilde\eps}'}

\def\square{{\vcenter{\vbox{\hrule height.3pt
         \hbox{\vrule width.3pt height5pt \kern5pt
            \vrule width.3pt}
         \hrule height.3pt}}}}

\eqnsection
\begin{document}
\title{H\"older continuity and occupation-time formulas for fBm self-intersection local time and its derivative}
\author{Paul Jung \quad Greg Markowsky}
\maketitle

\abstract{We prove joint H\"older continuity and an occupation-time formula for the self-intersection local time of fractional Brownian motion.  Motivated by an occupation-time formula, we also introduce a new version of the derivative of self-intersection local time for fractional Brownian motion and prove H\"older conditions for this process.  This process is related to a different version of the derivative of self-intersection local time studied by the authors in a previous work.}

\vspace{0.5cm}

{\bf Key words:} Intersection local time, Fractional Brownian motion, Occupation-time formula.

{\bf AMS Subject classification:} 60G22, 60J55

\section{Introduction}

The self-intersection local time of Brownian motion, formally defined as
\be \label{def:ILT}
\aa_t(y) := \int_0^t \int_0^s \dd(B_s -B_r -y)\, dr\, ds,
\ee
 was introduced in \cite{varad1969} and has since been well studied due to its various applications in physics, ranging from polymers to quantum field theory (see \cite{dynkin1988self, chen2010random} and the references therein). The self-intersection local time of fractional Brownian motion (FBM), $B^H_t$, was first investigated in \cite{rosen1987} in the planar case and was
further investigated, using tools from Malliavin calculus, by \cite{hu2001} and \cite{hu2005}. In particular, \cite{hu2005} showed its existence in dimension $d$ whenever the Hurst parameter
of FBM satisfies
$H<1/d$ (see \eqref{existence limit} below).

In the present work, we show that in one dimension, $\aa_t(y)$ for FBM is jointly H\"older continuous, in space and time, of any order below $1-H$. This result refines the fact that $\aa_t(y)$ is H\"older continuous in time of any order less than $1-H$ which can be derived from \cite[Thm 1.2]{xiao1997holder} along with the representation
\be\label{1.2}
\aa_t(y)=\frac{1}{2}\int_{\R} L^{x+y}_t L^{x}_t \, dx,
\ee
where $L^x_t$ is the local time of $B^H_t$ (Eq. \ref{1.2} also trivially shows $\aa_t(y)$ to be jointly continuous in space and time). We note that H\"older continuity, in time, of the intersection local time of independent FBMs has been investigated in \cite{wu2010regularity}; however,
techniques for self-intersections and independent intersections are typically different. For example, in the independent intersections case, \eqref{1.2} does not apply. Presently, our method of analysis boils down to showing that Kolmogorov's continuity criterion holds under various conditions.
An immediate by-product we obtain is an occupation-time formula which has previously shown only for $H=1/2$ \cite[Thm 1]{rosen2005}.

Let $\DD=\{(r,s): 0 < r < s < t\}.$
Motivated by spatial integrals with respect to local time which were developed in \cite{rogers1991local}, \cite{rosen2005} introduced a formal derivative (made rigorous) of $\aa_t(y)$ in the
one-dimensional Brownian case:
\be
\label{def:DILT}
\aa_t'(y) = -\int\!\int_\DD \dd'(B_s -B_r -y)\, dr\, ds.
\ee
This process has been further studied in \cite{markowsky2008renormalization, jung2012tanaka, hu2010central} as well as some of their references.
An FBM version of \eqref{def:DILT} was first considered in the works of \cite{yan2008, yan2009integration}.
Later, using a Tanaka formula as guiding intuition, \cite{jung2012tanaka} rigorously extended \eqref{def:DILT} to one-dimensional FBMs with $H<2/3$ by
\be \label{bsl2} \tilde\alpha_{t}'(y) := - \int\!\int_\DD
\dd'(B^H_s-B^H_r-y) (s-r)^{2H-1} \,dr \,ds.
\ee
An open problem stated in \cite{jung2012tanaka} was to prove the joint continuity, in space and time, of \eqref{bsl2}.

Here, we consider a different extension of \eqref{def:DILT} to the case of FBM which is guided by an occupation-time formula, rather than a Tanaka formula:
\be \label{bsl3}\dslt(y) := -\int\!\int_\DD
\dd'(B^H_s-B^H_r-y) \,dr \,ds.
\ee
Due to the absence of the kernel $(s-r)^{2H-1}$, we are able to show not only joint continuity, but also joint H\"older continuity of any order less than $1-2H, H<1/2$.


In preparation for our results, we set


\be
f_\eps(x) := \frac{1}{\sqrt{2\pi\eps}} e^{-\frac{1}{2}x^2/\eps} = \frac{1}{2\pi} \int_{\reals} e^{ipx} e^{-\eps p^2/2} dp,
\ee
and
\be \label{klam}
f_\eps'(x) := \frac{d}{dx} f_\eps(x) = \frac{-x}{\sqrt{2\pi \eps^3}} e^{-\frac{1}{2}x^2/\eps}=\frac{i}{2\pi} \int_{\reals} p e^{ipx} e^{-\eps p^2/2} dp.
\ee
The formal definitions \eqref{def:ILT} and \eqref{bsl3} are made rigorous whenever the following limits exist:
\bea
\label{existence limit}
\aa_t(y):=\lim_{\eps\to 0}\int\!\int_\DD
f_\eps(B^H_s-B^H_r-y) \,dr \,ds,\\
\dslt(y):=-\lim_{\eps\to 0}\int\!\int_\DD \label{existence limit2}
f'_\eps(B^H_s-B^H_r-y) \,dr \,ds.
\eea
In \cite{hu2005}, the above limit for $\aa_t(y)$ was shown to exist in $L^2(\Omega)$ for all $0<H<1$ (in dimension $d=1$).  For $\dslt(y)$ we have

\begin{proposition}\label{prop:existence}
If $H<2/3$, then the limit \eqref{existence limit2} exists  in $L^2(\Omega,\FF,\P)$.
\end{proposition}
The proof of the above proposition is an application of Lemma 3.1 in \cite{hu2001}, and can be deduced from the arguments of \cite{yan2008} with a slight correction given in \cite[Appendix]{jung2012tanaka}. It is therefore omitted. The following theorem now states the main results advertised in the title of this work:

\begin{theorem}\label{bg1}
Fix $0<H<1$ and suppose $t\in[0,T]$. The process $\aa_t(y)$ exists a.s. and in $L^p(\Omega)$ for all $p\in(0,\ff)$, and has a modification which is a.s. jointly  H\"{o}lder continuous in $(y,t)$ of any order less than $1-H$ and H\"older continuous in $y$ of any order
less than $\min(\frac{1}{H}-1,1)$. Also, if $g$ is continous, then
\begin{eqnarray} \label{oct1}
\int \! \int_\DD g(B^H_s -B^H_r)\, dr\, ds = \int_{\RR} g(y) \aa_t(y) dy,
\end{eqnarray}

If $H<1/2$, then  the process $\dslt(y)$ exists a.s. and in $L^p(\Omega)$ for all $p\in(0,\ff)$ and has a modification which is a.s. jointly  H\"{o}lder continuous in $(y,t)$ of any order less than $1-2H$ and H\"older continuous in $y$ of
any order less than $\min(1/H-2,1)$.
Moreover,  $\dslt(y)= \frac{d}{dy} \aa_t(y)$.
Finally, for any $g\in C^1$
\begin{eqnarray}
\int \! \int_\DD g'(B^H_s -B^H_r)\, dr\, ds = -\int_{\RR} g(y) \dslt(y) dy. \label{occtim}
\end{eqnarray}
\end{theorem}

It will be clear from the proof of the above that any discontinuities of $\dslt$ for $H\ge 1/2$ can only occur at $y=0$. We next present evidence that there are indeed discontinuities at $y=0$ when $H\ge 1/2$. We also give the ``expected''  behavior of $\dslt(y)$ as $y\to 0$.

\begin{proposition} \label{exproposition} For all $0<H<2/3$,
$\E[\dslt(y)]$ is continuous in $y$ for $y \neq 0$.  It is also continuous at $y=0$ if and only if $H<1/2$. Moreover, at $y=0$:

\begin{itemize} \label{}

\item[(i)] \lll{disney} If $1/3< H< 2/3$, then $$\lim_{y \lar 0} \frac{\E[\dslt(y)]\sgn(y)}{|y|^{{1/H}-2}}= \frac{-t}{\sqrt{2\pi}} \int_{0}^{\ff} v^{-3H}\exp(-v^{-2H}/2) dv.$$

\item[(ii)] If $H = 1/3$, then $$\lim_{y \lar 0} \frac{\E[\dslt(y)]}{y\log{|y|}}=\frac{3t}{\sqrt{2\pi}}.$$

\item[(iii)] If $0<H < 1/3$, then $$\lim_{y \lar 0} \frac{\E[\dslt(y)]}{y}=\frac{-t^{2-3H}}{(1-3H)\sqrt{2\pi}}.$$

\end{itemize}
\end{proposition}

The behavior of $\dslt(y)$ as $y \lar 0$ for $1/2 \leq H < 2/3$ is quite interesting. Proposition \ref{exproposition} shows that continuity cannot be expected at $y=0$. On the other hand, in the Brownian case, $H=1/2$, \cite{rosen2005} showed that the renormalization $\aaa'_{t}(y) - t \E[\aaa'_{t}(y)]$ is continuous at $y=0$. This follows readily from the Tanaka formula for $\aaa'_t(y)$ as well (see \cite{markowsky2008proof}). We venture this as a conjecture for $1/2<H<2/3$.

\vski

\noi {\textbf{Conjecture}: \it The renormalization $\dslt(y) - t \E[\dslt(y)]$ has a modification which is continuous in $y$ for $1/2<H<2/3$.}

In the next section, we provide proofs for Proposition \ref{exproposition} and Theorem \ref{bg1}, while in the appendix we discuss the reasoning behind the conjecture, as well as a possible road map for its proof.

\section{Proofs}

\begin{proof}[Proof of Proposition \ref{exproposition}]
Let $\dslteps(y)$ be defined by replacing the delta distribution in \eqref{def:ILT} with $f_\eps$.
By symmetry, $\E[\dslteps(0)]=0$. For $y \neq 0$, using \rrr{klam} we have

\begin{equation} \label{presque}
\begin{split}
\E[\dslteps(y)] & = \frac{-i}{2\pi} \int_{0}^{t} \int_{0}^{s} \int_{\RR} p e^{-ipy} e^{-\eps p^2/2}  \E[e^{ip(B^H_s-B^H_r)}] \,dp\,dr\,ds \\
& = \frac{-i}{2\pi} \int_{0}^{t} \int_{0}^{s} \int_{\RR} p e^{-ipy} e^{-p^2(\eps + (s-r)^{2H})/2} \,dp\,dr\,ds \\
& = \frac{-y}{\sqrt{2\pi}} \int_{0}^{t} \int_{0}^{s} \frac{e^{-{y^2}/[2(\eps + (s-r)^{2H})]}}{(\eps + (s-r)^{2H})^{3/2}} \,dr\,ds.
\end{split}
\end{equation}

For $y\neq 0$ one may, by dominated convergence and convergence in $L^2(\Omega)$, set $\eps=0$ above to obtain $\E[\dslt(y)]$.  It is then clear that $\E[\dslt(y)]$ is continuous at all $y\neq 0$. If $y=0$,
$\E[\dslteps(y)]\equiv 0$, and since $\dslt(y)$ converges in $L^2(\Omega)$, it must be that
$\E[\dslt(y)]=0$. Since \rrr{presque} implies $\E[\dslt(-y)]=-\E[\dslt(y)]$, it suffices to determine the behavior of $\E[\dslt(y)]$ as $y \searrow 0$ .

Suppose $y>0$. Setting $\eps=0$ and replacing $s-r$ by $u$ and then $u$ by $y^{1/H}v$ in \eqref{presque} leads to

\bea \label{crack0}
\E[\dslt(y)] &=& \frac{-y}{\sqrt{2\pi}} \int_{0}^{t} \int_{0}^{s} \frac{e^{-y^2/(2u^{2H})}}{u^{3H}} du ds\\
&=& \frac{-y^{1/H - 2}}{\sqrt{2\pi}} \int_{0}^{t} \int_{0}^{s/y^{1/H}} \frac{e^{-1/(2v^{2H})}}{v^{3H}} dv ds.\nn
\eea
When $H>1/3$, the inner integral converges as $y\searrow 0$, so we get
\bea
\lim_{y\searrow 0}\frac{\E[\dslt(y)] }{y^{1/H - 2}}&=& \frac{-1}{\sqrt{2\pi}} \int_{0}^{t} \int_{0}^{\ff} \frac{e^{-1/(2v^{2H})}}{v^{3H}} dv ds\\
&=& \frac{-t}{\sqrt{2\pi}} \int_{0}^{\ff} \frac{e^{-1/(2v^{2H})}}{v^{3H}} dv\nn
\eea

For the cases $H\le 1/3$, use Fubini on \eqref{crack0} and set $y=M^{-H}$ to get
\bea\label{crack}
\E[\dslt(y)] &=& \frac{-y^{1/H - 2}}{\sqrt{2\pi}} \int_{0}^{t/y^{1/H}} \int_{y^{1/H}v}^{t} \frac{e^{-1/(2v^{2H})}}{v^{3H}} ds dv\\
&=& \frac{-M^{2H - 1}}{\sqrt{2\pi}} \int_{0}^{tM} \int_{v/M}^{t} \frac{e^{-1/(2v^{2H})}}{v^{3H}} ds dv.\nn
\eea
By  L'H\^{o}pital's rule, for $H<1/3$,
\bea\label{crack1}
\lim_{y\searrow 0}\frac{\E[\dslt(y)]}{y} &=& \lim_{M\to\ff}\frac{-\int_{0}^{tM} \int_{v/M}^{t} \frac{e^{-1/(2v^{2H})}}{v^{3H}} ds dv}{M^{1-3H}\sqrt{2\pi} }\\
&=& \lim_{M\to\ff}\frac{-\frac{d}{dM}\int_{0}^{tM} t \frac{e^{-1/(2v^{2H})}}{v^{3H}} dv}{(1-3H)M^{-3H}\sqrt{2\pi} }\nn\\
&=& \lim_{M\to\ff}\frac{-t^{2-3H} e^{-1/(2(tM)^{2H})}}{{(1-3H)\sqrt{2\pi}}} = \frac{-t^{2-3H}}{(1-3H)\sqrt{2\pi}}\nn.
\eea
Similarly, for $H=1/3$,
\bea\label{crack2}
\lim_{y\searrow 0}\frac{\E[\dslt(y)]}{y \log(y)} &=& \lim_{M\to\ff}\frac{3t^2\frac{e^{-1/(2(tM)^{2/3})}}{tM}}{\frac{d}{dM}\log(M)\sqrt{2\pi} } = \frac{3t}{\sqrt{2\pi}}.
\eea
\end{proof}

\begin{proof}[Proof of Theorem \ref{bg1}]
In the following proof, $C$ denotes a constant which may change from line to line, and $t$ will always be bounded by a fixed $T>0$.

We first consider the results about $\dslt$ for $H<1/2$. For the existence of this process a.s. and  in $L^p(\Omega)$ for all $p\in(0,\ff)$,
by Kolmogorov's continuity criterion \cite[Thm I.2.1]{revuz1999}, it suffices to
 show
\begin{equation} \label{rundmc0}
\E[|\dslteps(y)- \dsltepse(y)|^n] \leq C|\eps-\tilde{\eps}|^{n\la} \ \ \text{ for any } \lambda\in[0,1].
\end{equation}
Using \rrr{klam} we have for any  $\lambda\in [0,1]$
\begin{eqnarray}
\label{try}
&&\E[|\dslteps(y)-\dsltepse({y})|^n] \\
&=& \nn\frac{1}{(2\pi)^n} \Big|\int_{\DD^n} \int_{\RR^n} \E[\prod_{k=1}^n e^{ip_k(B^H_{s_k}-B^H_{r_k})}] \prod_{k=1}^n p_k (e^{-\eps p_k^2/2}-e^{-\eps p_k^2/2})e^{ip_ky} d\vec{p}d\vec{r}d\vec{s} \Big|\\ \nn
&\le& C |\eps-\tilde{\eps}|^{n\la} \int_{\DD^n} \int_{\RR^n} \Big|\E[\prod_{k=1}^n e^{ip_k(B^H_{s_k}-B^H_{r_k})}]\Big|  \prod_{k=1}^n|p_k|^{1+2\la}  d\vec{p}d\vec{r}d\vec{s} .
\end{eqnarray}
where in the inequality we have used the bound
\begin{eqnarray} \label{bnd77} 
|e^{-\eps p^2/2} - e^{-\tilde{\eps} p^2/2}| \leq C |p|^{2\la} |\eps-\tilde{\eps}|^{\la} \ \ \text{ for any } \lambda\in[0,1].
\end{eqnarray}
We need only show now that the the integral on the right side of \eqref{try} converges. To do this we estimate the two products in the integrand.

Let us first consider the product inside the expectation.
This expectation in the integrand will take different forms over different regions of integration, depending on the ordering of the $r_k$'s and $s_k$'s. Fix such an ordering and let $\ell_1 \leq \ell_2 \leq \ldots \leq \ell_{2n}$ be a relabeling of the set $\{r_1,s_1,r_2,s_2, \ldots, r_n,s_n\}$. We may then write

\begin{equation} \label{kidr}
\prod_{k=1}^n e^{ip_k(B^H_{s_k}-B^H_{r_k})} = \prod_{j=1}^{2n-1} e^{iu_j(B^H_{\ell_{j+1}}-B^H_{\ell_{j}})},
\end{equation}
where the $u_j$'s are properly chosen linearly combinations of the $p_k$'s to make \rrr{kidr} an equality. A visual device which may aid, due to Rosen (personal communication), is as follows. We draw an arc corresponding to each $p_k$ and whose endpoints are in the correct order. The intervals between endpoints belong to
$u_j$'s, and each $u_j$ is the linear combination of the $p_i$'s which arch over it. As an example, let us suppose $n=6$ and $r_1<r_2<s_2<r_3<r_4<s_1<s_3<r_5<s_4<s_5<r_6<s_6$. The
picture associated to this configuration is

\vspace{-1 in}
\hspace{-1cm}
\includegraphics[width=135mm, height=105mm]{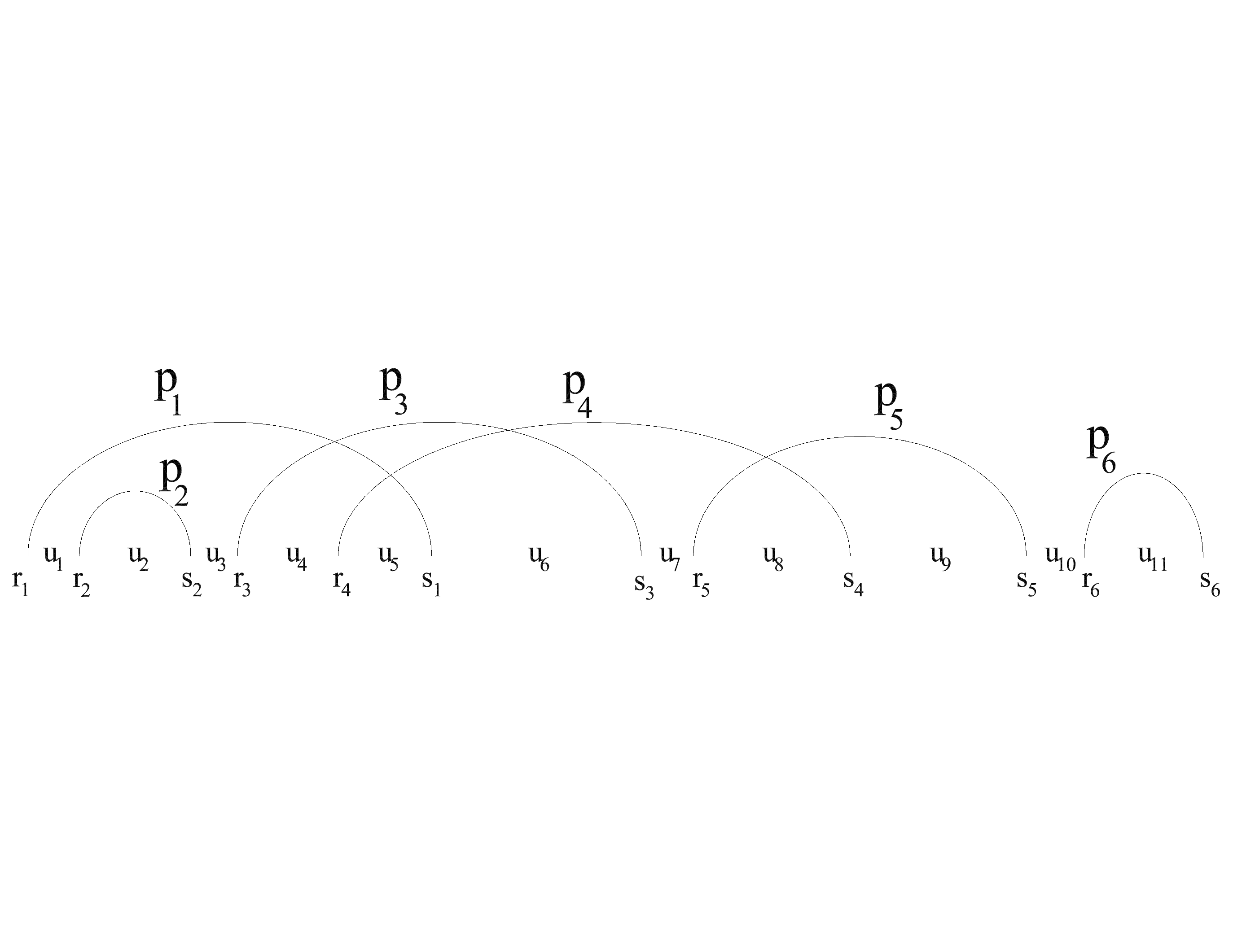}

\vspace{-1.5 in}
\begin{center}Figure 1\end{center}

From this picture we see easily that $u_1 = p_1$, $u_2 = p_1+p_2$,
$u_3 = p_1$, $u_4 = p_1 + p_3$, $u_5 = p_1+p_3+p_4$, $u_6 = p_3
+ p_4$, $u_7 = p_4$, $u_8=p_4+p_5$, $u_9 = p_5$, $u_{10} = 0$, and $u_{11}=p_6$. Written in this form we may appeal to the local nondeterminism of FBM (see for instance \cite[Lemma 8.1]{berman1974} or \cite[Eq. 2.3]{rosen1987}) to conclude that
\begin{equation} \label{van}
\E|\prod_{k=1}^n e^{ip_k(B^H_{s_k}-B^H_{r_k})}| \leq e^{-C \sum_{j=1}^{2n-1} |u_j|^2 (\ell_{j+1}-\ell_j)^{2H}}.
\end{equation}

Let us now bound the second product in the integrand of \eqref{try}.
Note that there exist $j_1, j_2$ such that  $p_k=u_{j_1} - u_{j_1-1}$ and $p_k=u_{j_2-1} - u_{j_2}$ for appropriate choices of $j_1, j_2$ (where we set $u_0 =u_{2n} = 0$). Thus
\begin{equation} \label{steam}
\begin{split}
|p_k|^{1+2\la} & = |u_{j_1} - u_{j_1-1}|^{\frac{1+2\la}{2}} |u_{j_2-1} - u_{j_2}|^{\frac{1+2\la}{2}} \\
& \leq C (|u_{j_1}|^{\frac{1+2\la}{2}}+|u_{j_1-1}|^{\frac{1+2\la}{2}})(|u_{j_2}|^{\frac{1+2\la}{2}}+|u_{j_2-1}|^{\frac{1+2\la}{2}}).
\end{split}
\end{equation}

Setting $a_j=\ell_{j+1}-\ell_j$ (with $\ell_0=0$) and using \rrr{van} and \rrr{steam} we can now bound the integral on the right side of \eqref{try} by
\bea\label{tryq}
&&C\int_{\RR^n} \int_{[0,t]^{2n}}  \prod_{j=1}^{2n-1} e^{-C |u_j|^2 a_j^{2H}} \prod_{j=0}^{2n}  (|u_{j}|^{\frac{1+2\la}{2}}+|u_{j-1}|^{\frac{1+2\la}{2}}) d\vec{a} d\vec{p} \nn \\
&\leq& C\int_{\RR^n} \frac{\prod_{j=0}^{2n}  (|u_{j}|^{\frac{1+2\la}{2}}+|u_{j-1}|^{\frac{1+2\la}{2}})}{\prod_{j=1}^{2n-1}(1+|u_j|^{1/H})}d\vec{p},
\eea
where we have used the simple bound 
\begin{equation} \label{clay}
\int_{0}^{t} e^{-C |u|^2 a^{2H}} da \leq \frac{C}{1+|u|^{1/H}}.
\end{equation}
Expanding the product in the numerator of  \eqref{tryq} gives us the sum of a number of terms of the form $$\prod_{j=1}^{2n-1} |u_{j}|^{\frac{(1+2\la)m_j}{2}}$$ where $m_j = 0,1,$ or $2$ (terms containing $|u_0|$ or $|u_{2n}|$ are equal to $0$). We may therefore reduce our problem to showing that

\begin{equation} \label{gre}
\int_{\RR^n} \frac{d\vec{p}}{\prod_{j=1}^{2n-1}(1+|u_j|^{\frac{1}{H}-\frac{(1+2\la)m_j}{2}})}
\end{equation}
is finite. We perform a linear transformation changing this into an integral with respect to variables $u_{k_1}, u_{k_2}, \ldots, u_{k_n} \in \{u_1, \ldots, u_{2n-1}\}$ which span $\{p_1, \ldots, p_n\}$ in order to bound \rrr{gre} by

\begin{equation} \label{ek}
\begin{split}
C& \int_{\RR^n} \frac{d\vec{u}}{\prod_{j=1}^{2n-1}(1+|u_{k_j}|^{\frac{1}{H}-\frac{(1+2\la)m_j}{2}})} \\
& \leq C\int_{\RR^n} \frac{d\vec{u}}{\prod_{j=1}^{2n-1}(1+|u_{k_j}|^{\frac{1}{H}-(1+2\la)})}.
\end{split}
\end{equation}
This is finite if we choose $\la$ so that $\frac{1}{H}-(1+2\la) > 1$.

Let us next prove the H\"older continuity results for $\dslt$.  This is established by once again using  Kolmogorov's continuity criterion.  Since
$(a+b)^n$ and $a^n+b^n$ are equivalent up to constants for $a,b>0$,
it suffices to check that for $\la<\min(1/H-2,1)$ and $\bb<1-2H$,
\begin{equation} \label{rundmc}
\begin{gathered}
\E[|\dslt(y)- \dslt(\tilde{y})|^n] \leq C|y-\tilde{y}|^{n\la} \  \text{ and} \\
\E[|\dslt(y)- \hat\aa_t'(y)|^n] \leq C|t-\tilde{t}|^{n\bb}.
\end{gathered}
\end{equation}

Following a similar calculation to \rrr{try} at $\eps=0$ (this is possible since we have established convergence in $L^p(\Omega)$ for all $p\in(0,\ff)$),
we use \rrr{klam} to get that for any  $\lambda\in [0,1]$
\begin{eqnarray}
\label{try200}
&&\E[|\dslt(y)-\dslt(\tilde{y})|^n] \\
&\le& C |y-\tilde{y}|^{n\la} \int_{\DD^n} \int_{\RR^n} \Big|\E[\prod_{k=1}^n e^{ip_k(B^H_{s_k}-B^H_{r_k})}]\Big|  \prod_{k=1}^n|p_k|^{1+\la}  d\vec{p}d\vec{r}d\vec{s} .\nn
\end{eqnarray}
where we have used the bound
\begin{equation} \label{bnd76}
|e^{ip_ky}-e^{ip_k\tilde{y}}| \leq C|p_k|^\la|y-\tilde{y}|^\la\ \ \text{ for any } \lambda\in [0,1].
\end{equation}
Following \eqref{van}-\eqref{ek} (with $\lambda$ instead of $2\la$), we establish the $y$-variation in \eqref{rundmc}.

For the $t$-variation, set $\tilde\DD=\{(r,s): 0 < r < s < \tilde t\}$ and assume $\tilde{t} > t$. Once again using \eqref{klam} and  \eqref{van}-\eqref{ek} (with $\lambda=0$ now), we get
\begin{eqnarray} \label{try100}
&&\E[|\hat{\aa}_{\tilde t}'(y)-\dslt(y)|^n] \\\nn&=& \frac{1}{(2\pi)^n} \Big| \int\!\int_{(\tilde{\DD}\backslash\DD)^n} \int_{\RR^n} \E[\prod_{k=1}^n e^{ip_k(B_{s_k}-B_{r_k})}] \prod_{k=1}^n p_k e^{ip_ky} d\vec{p}d\vec{r}d\vec{s} \Big| \\ \nn
&=& \frac{1}{(2\pi)^n} \Big|\int_{[t,\tilde{t}]^n} \int_{[0,s_1] \times \ldots [0,s_n]} \int_{\RR^n} \E[\prod_{k=1}^n e^{ip_k(B_{s_k}-B_{r_k})}] \prod_{k=1}^n p_k e^{ip_ky} d\vec{p}d\vec{r}d\vec{s} \Big|\\
&\leq& C\nn \int_{\RR^n} \prod_{j=0}^{2n}  (|u_{j}|^{\frac{1}{2}}+|u_{j-1}|^{\frac{1}{2}})  \int_{[0,\tilde t]^{2n}}  \prod_{j=1}^{2n-1} e^{-C |u_j|^2 a_j^{2H}} \prod_{k=1}^{n} 1_{[t,\tilde{t}]}(s_k) d\vec{r} d\vec{s} d\vec{p}.
\end{eqnarray}
We apply H\"older's inequality with $\ga = 1 - \bb$ to get a bound on the right-most integral:
\begin{equation} \label{try25}
\begin{split}
\int_{[0,\tilde t]^{2n}}&  \prod_{j=1}^{2n-1} e^{-C |u_j|^2 a_j^{2H}} \prod_{k=1}^{n}1_{[t,\tilde{t}]}(s_k) d\vec{r}d\vec{s} \\
& \leq \Big(\int_{[0,\tilde t]^{2n}}  \prod_{j=1}^{2n-1} e^{-C |u_j|^2 a_j^{2H}}\prod_{k=1}^{n}d\vec{r}d\vec{s}\Big)^\ga \Big(\int_{[0,t]^{2n}}  \prod_{k=1}^{n} 1_{[t,\tilde{t}]}(s_k) d\vec{r}d\vec{s}\Big)^\bb \\
& \qquad \qquad \quad \leq C |t-\tilde{t}|^{n\bb} \prod_{j=1}^{2n-1}  \frac{1}{(1+|u_j|^{1/H})^\ga}
\end{split}
\end{equation}
which gives us
\begin{equation} \label{try3}
\begin{split}
\E[|&\dslteps(y)-\dsltepst(y)|^n] \leq \\ & C |t-\tilde{t}|^{n\bb} \int_{\RR^n} \frac{\prod_{j=0}^{2n}  (|u_{j}|^{\frac{1}{2}}+|u_{j-1}|^{\frac{1}{2}})}{\prod_{j=1}^{2n-1}  (1+|u_j|^{1/H})^\ga}  d\vec{p}.
\end{split}
\end{equation}
Following steps \rrr{gre} and \rrr{ek} bounds the above by
\begin{equation} \label{ek2}
\begin{split}
C& \int_{\RR^n} \frac{d\vec{u}}{\prod_{j=1}^{2n-1}(1+|u_{k_j}|^{\frac{\ga}{H}-\frac{m_j}{2}})} \\
& \leq C\int_{\RR^n} \frac{d\vec{u}}{\prod_{j=1}^{2n-1}(1+|u_{k_j}|^{\frac{\ga}{H}-1})}.
\end{split}
\end{equation}
Choosing $\bb < 1 - 2H$ implies $\ga = 1-\bb > 2H$ which in turn implies convergence of the above integral. This establishes \rrr{rundmc} and completes the proof of H\"older continuity for $\dslt$.

We finish our proof of the part of the theorem concerning $\dslt$ by showing \eqref{occtim} (saving $\frac{d}{dy}\aa_t = \dslt$ for later). Following the arguments of \cite{rosen2005}, the first observation is that
\eqref{rundmc} is valid, not only for $\dslt$, but for $\dslteps$ as well.  Thus, by the H\"older continuity in $(\eps,y,t)$ implied by Kolmogorov's continuity criterion, the convergence
of $\dslteps$ as $\eps\to 0$ is locally uniform in $y$ which  implies
 $L^1_\text{loc}$-convergence as well. This allows the following manipulations for any $g\in C^1$ with compact support:
\begin{equation} \label{wilhor}
\begin{split}
\int_{\RR} g(y) \dslt(y) dy & = \lim_{\eps \lar 0} \int_{\RR} g(y) \dslteps(y)\, dy \\
& =-\lim_{\eps \lar 0} \int_{\RR} g(y) \Big( \int \! \int_{\DD}f_\eps'(B^H_s-B^H_r - y)\,dr\,ds \Big)\, dy \\
& = -\lim_{\eps \lar 0} \int \! \int_{\DD} \int_{\RR} g'(y)f_\eps(B^H_s-B^H_r - y) \,dy\,dr\,ds \\
& = -\lim_{\eps \lar 0} \int \! \int_{\DD} f_\eps * g'(B^H_s-B^H_r)\, dr\,ds \\
& = -\int \! \int_{\DD} g'(B^H_s-B^H_r) \,dr\,ds,
\end{split}
\end{equation}
where the negative sign from the integration by parts  in the third identity cancels the negative sign of the chain rule in the $y$ derivative of $f_\eps$. The result is extended to all $g\in C^1$ by noting that $B^H_s$ is a.s. bounded on $[0,t]$ hence $\dslt$ has a.s. compact support.

The H\"older continuity arguments for $\dslt(y)$ apply to $\aa_t(y)$ with the only difference between the two proofs being the presence of the $\prod_{k=1}^n p_k$ term in the integrand. Following steps \rrr{try} through \rrr{ek} leads to
\begin{equation} \label{ekek}
\E[|\aa_{\eps, t}(y)-\aa_{\eps, t}(\tilde{y})|^n] \leq C|y-\tilde{y}|^{n\la} \int_{\RR^n} \frac{\prod_{j=1}^{n}du_{k_j}}{\prod_{j=1}^{2n-1}(1+|u_{k_j}|^{\frac{1}{H}-\la})}.
\end{equation}
This integral is finite, since $\frac{1}{H}-\la > 1$. The variation in $\eps$ is handled in the same way. Kolmogorov's criterion applies as before. As mentioned in the introduction, H\"older continuity in $t$ for $\aa_t(y)$ follows from \eqref{1.2} and \cite{xiao1997holder}.

%
%

The occupation times formula \rrr{oct1} is derived similarly to \rrr{wilhor}, minus the application of integration by parts. Finally, we apply the locally uniform convergence to $\frac{d}{dy}\aa_{t,\eps} = \dslteps$ to obtain
\be
\aa_t(y) = \aa_t(x) +\int_x^y\dslt(u) du
\ee
which gives $\frac{d}{dy}\aa_t = \dslt$.
\end{proof}

\section*{Appendix: Continuity for $A\subset\DD$ when $H>1/2$}

In the introduction we conjectured that
\be\label{renormalzation}
\dslt(y) - t \E[\dslt(y)]
\ee has a spatially continuous modification for $1/2<H<2/3$. Rosen \cite{rosen2005} proved this in the case $H=1/2$ by considering
a generalization of \eqref{bsl3} to subsets $A\subset\DD$.  More specifically, define
\be \label{bsl100}\dsltA := -\lim_{\eps \lar 0} \int\!\int_A
f_\eps'(B^H_s-B^H_r-y) \,dr \,ds,
\ee
if the limit exists. We will begin by proving the following, which shows that $\dsltA$ exists when $A$ is sufficiently bounded away from the diagonal set $\{0<r=s<t\}$.

\begin{proposition} \label{bg3} The following hold.

\begin{itemize}
\item[(i)] If $A \subseteq \DD_\kappa := \{0 < r < s-\kappa < t-\kappa \}$, then $\aaa'_t(y,A)$ exists and is jointly H\"older continuous in $y$ and $t$.

\item[(ii)] If $A \subseteq A_1^1 := [0,1/2] \times [1/2,1]$, then $\aaa'_t(y,A)$ exists and is jointly H\"older continuous in $y$ and $t$.
\end{itemize}

\vski

\noi In both cases, $\aaa'_t(y,A)$ can be taken to be H\"older continuous of any order $\la < \frac{1}{H} - \frac{3}{2}$ in $y$ and of any order $\bb < 1- \frac{3}{2}H$ in $t$, and we also have, for any $g\in C^1$,
\begin{eqnarray}
\int \! \int_A g'(B^H_s -B^H_r)\, dr\, ds = -\int_{\RR} g(y) \dslt(y,A) dy. \label{occtim}
\end{eqnarray}

\end{proposition}

\begin{proof}[Proof of Proposition \ref{bg3}(i)] As before, we will show that \rrr{rundmc} holds. Again we begin with

\begin{equation} \label{try7}
\begin{split}
E&[|\aa_{\eps, t}'(y,A)-\aa_{\eps, t}'(y',A)|^n] = \\ & \frac{1}{(2\pi)^n} \Big|\int_{A^n} \int_{\RR^n} E[\prod_{k=1}^n e^{ip_k(B_{s_k}-B_{r_k})}] \prod_{k=1}^n p_k (e^{ip_ky}-e^{ip_ky'}) e^{-\eps p_k^2/2} dp_kdr_kds_k \Big|.
\end{split}
\end{equation}
The expectation in the integrand depends again upon the ordering of the $s_k$'s and $r_k$'s, but unlike in the earlier results we will need to do a careful analysis of different possible orderings in order to obtain the required convergence. Of particular interest will be configurations which contain {\it isolated intervals}, that is, values $k'$ and their corresponding variables $p_{k'}$ for which the interval $[r_{k'},s_{k'}]$ contains no other $r_k$ or $s_k$. For instance, in Figure 1 above $p_2$ and $p_6$ correspond to isolated intervals, while no others do. As with a number of other cases in which this general method has been applied, configurations with isolated intervals present special difficulties(for examples, see \cite{markowsky2008renormalization} and \cite{rosen1988}). We will therefore begin by supposing that we have a configuration of $\{r_1,s_1, \ldots ,r_n,s_n\}$ which contains no isolated intervals. We may follow the argument for the case $H<1/2$ in order to reduce our problem again to showing

\begin{equation} \label{try8}
C \int_{\RR^n} \frac{\prod_{k=1}^{n}dp_k}{\prod_{j=1}^{2n-1}(1+|u_j|^{\frac{1}{H}-\frac{(1+\la)m_j}{2}})} < \ff
\end{equation}
It is clear that if $\TT=\cup_{k=1}^n (r_k,s_k)$ is not connected then the integral factors into the product over the different components, so we may assume that $\TT$ is connected. Suppose that we can find two sets $\AA, \BB \subseteq \{u_1, \ldots , u_{2n-1}\}$ with the properties that $\AA$ and $\BB$ both span $\{p_1, \ldots , p_n\}$ and if $u_j \in \AA \cap \BB$ then $m_j \leq 1$. We can then use the Cauchy-Schwarz inequality to bound \rrr{try8} by

\begin{equation} \label{page}
\begin{split}
\Big(\int_{\RR^n} & \frac{\prod_{k=1}^{n}dp_k}{\prod_{j=1}^{2n-1}(1+|u_j|^{\frac{1}{H}-\frac{(1+\la)m_j}{2}})}\Big)^2 \leq \\
& \int_{\RR^n} \frac{\prod_{k=1}^{n}dp_k}{\prod_{u_j \in \AA \cap \BB}(1+|u_j|^{\frac{1}{H}-\frac{(1+\la)m_j}{2}}) \prod_{u_j \in \AA \backslash \BB}(1+|u_j|^{\frac{1}{H}-\frac{(1+\la)m_j}{2}})^2} \\
& \times \int_{\RR^n} \frac{\prod_{k=1}^{n}dp_k}{\prod_{u_j \in \AA \cap \BB}(1+|u_j|^{\frac{1}{H}-\frac{(1+\la)m_j}{2}}) \prod_{u_j \in \BB \backslash \AA}(1+|u_j|^{\frac{1}{H}-\frac{(1+\la)m_j}{2}})^2} \\
& \leq C \int_{\RR^n} \frac{\prod_{u_j \in \AA}du_j}{\prod_{u_j \in \AA}(1+|u_j|^{d(H,\la)})} \int_{\RR^n} \frac{\prod_{u_j \in \BB}du_j}{\prod_{u_j \in \BB}(1+|u_j|^{d(H,\la)})},
\end{split}
\end{equation}
where $d(H,\la) = \min(\frac{1}{H}-\frac{(1+\la)}{2}, 2(\frac{1}{H}-(1+\la)))$. The assumption that $\la < \frac{1}{H} - \frac{3}{2}$ implies that $d(H,\la)>1$, so the final expression in \rrr{page} is finite. We need therefore only show that we can always find the sets $\AA, \BB$ with the necessary properties. To do this, we will need to introduce a bit of terminology which was utilized in \cite{markowsky2008renormalization}. For each $j$, either $u_{j} - u_{j-1} = p_k$ or $u_{j} - u_{j-1} = -p_k$ for some $k$. In the first case we will refer to $u_j$ as
{\it increasing} and in the second case we will say that $u_j$ is {\it decreasing}. We will call a variable $p_{k'}$ {\it s-free}
if there is no $s_k$ contained in $(r_{k'},s_{k'})$, and similarly $p_{k'}$ is {\it r-free}
if there is no $r_k$ contained in $(r_{k'},s_{k'})$. For example, given the configuration indicated in Figure 1, the variables $u_1, u_2, u_3,$ and $u_6$ are all increasing, while $u_4, u_5,$ and $u_7$ are decreasing. $p_1$ is the only s-free variable, and $p_4$ is the only r-free variable. The following lemma was stated in \cite{rosen1988}; a simple proof using the terminology in this paper can be found in \cite{markowsky2008renormalization}.

\begin{lemma} \label{linal}
\begin{itemize} \label{}

\item[(i)] The span of the increasing $u_j$'s is equal to the span of the set of
all $p_k$'s which are not r-free. Furthermore, suppose that for each
r-free $p_k$ we choose $u(p_k)$ to be any one of the decreasing
$u_j$'s which contains $p_k$ as a term. Then, if we let $\AA = \{$set
of increasing $u_j$'s$\} \unn \{$set of all $u(p_k)$'s$\}$, $\AA$
spans the entire set $\{ p_1, ..., p_n \}$.

\item[(ii)] The span of the decreasing $u_j$'s is equal to the span of the set of
all $p_k$'s which are not s-free. Furthermore, suppose that for each
s-free $p_k$ we choose $u(p_k)$ to be any one of the increasing
$u_j$'s which contains $p_k$ as a term. Then, if we let $\BB = \{$set
of decreasing $u_j$'s$\} \unn \{$set of all $u(p_k)$'s$\}$, $\BB$
spans the entire set $\{ p_1, ..., p_n \}$.

\end{itemize}
\end{lemma}
In fact, only $(ii)$ was proved in \cite{markowsky2008renormalization}, but $(i)$ follows by symmetric arguments. This lemma allows us to form the sets $\AA$ and $\BB$ which have the desired properties, provided that we can always choose the $u(p_k)$'s for the r-free and s-free variables to be equal to $u_j$'s for which $m_j \leq 1$. However, this is certainly possible, since each r-free and s-free variable appears as a term in at least two consecutive $u_j$'s(it is here that we use the assumption of no isolated intervals), and the definition of the $m_j$'s show that it is impossible to have $m_j=m_{j+1}=2$. This gives us the required control over in the variation in $y$. The same technique applies to control the variation in $\eps$. For the variation in $t$ we follow steps \rrr{try100} to \rrr{ek2}. We must therefore show that

\begin{equation} \label{ek3}
C \int_{\RR^n} \frac{\prod_{j=1}^{n}du_{k_j}}{\prod_{j=1}^{2n-1}(1+|u_{k_j}|^{\frac{\ga}{H}-\frac{m_j}{2}})}
\end{equation}
is finite, where $\ga = 1 - \bb$. As before we apply Lemma \ref{linal} to form the sets $\AA$ and $\BB$ and follow the manipulations in \rrr{page} in order to show that \rrr{ek3} is finite. This requires $\min(\frac{\ga}{H}-\frac{1}{2}, 2(\frac{\ga}{H}-1))$ to be at least $1$, which holds due to our requirement that $\bb < 1 - \frac{3}{2}H$. This completes the proof in the case of no isolated intervals.

Let us now suppose that isolated intervals are present. Essentially our task is to "remove the isolated intervals"\footnotemark \footnotetext{This process is described in greater detail than here in \cite{markowsky2008renormalization}, with a picture included.}, that is, to integrate out the variables corresponding to isolated intervals in order to reduce a configuration to a smaller one. Note that our restrictions on $A$ imply that if $[r_k,s_k]$ is an isolated interval then $s_k > r_k + \ka$. It is this separation which will allow us to substitute the bound

\begin{equation} \label{clay2}
\int_{\ka}^{t} e^{-C |u_j|^2 a^{2H}} da \leq \frac{Ce^{-C\ka^{2H}|u_j|^2}}{1+|u_j|^{1/H}}.
\end{equation}
in place of \rrr{clay2} when $u_j$ is the lone $u$ containing the isolated $p_k$ as a term. In order to bound the variation in $y$ and $t$, we follow the steps \rrr{try} through \rrr{ek} in order to reduce our problem to showing that

\begin{equation} \label{ek66}
\int_{\RR^n} \frac{\prod_{j \in \II} e^{-C\ka^{2H}|u_{k_j}|^2}\prod_{j=1}^{n}du_{k_j}}{\prod_{k=1}^{2n-1}(1+|u_{k}|^{\frac{1}{H}-\frac{(1+\la)m_k}{2}})}
\end{equation}
is finite, where $\II \subseteq \{1,2, \ldots, n\}$ is the set of all $j$ such that $u_{k_j}$ contains an isolated $p$ value. We may now integrate with respect to the $u_{k_j}$'s in $\II$ and obtain a constant bound. What remains is an integral corresponding to a smaller configuration of $p_k$'s but with several extra powers of $|u_k|$'s in the bottom; this occurs since if $u_{k_j}$ contains an isolated $p$ value then $u_{k_j-1} = u_{k_j+1}$. If the new configuration contains any isolated intervals itself, this "doubling" of the powers in the denominator allows us to remove the isolated intervals of the new configuration, since we will have a term at least as convergent as $(1+|u_{k}|^{2(\frac{1}{H}-(1+\la))})$ in the denominator, and $2(\frac{1}{H}-(1+\la))$ is greater than 1 by our assumption on $\la$. This reduces the configuration to a still smaller one. We may continue in this manner until we have either integrated out all of the $u$'s, in which case we are done, or we have arrived at a configuration with no isolated intervals. In the latter case we may then appeal to the work done for the no isolated intervals case, and again we are done. The same process works for the variation in $t$, except that the isolated intervals in the later configurations now have a term at least as convergent as $(1+|u_{k}|^{2(\frac{\ga}{H}-1)})$ in the denominator. Again, the exponent $2(\frac{\ga}{H}-1)$ is at least 1 due to our assumption on $\bb$. This completes the proof of Proposition \ref{bg3} (i).
\end{proof}

\vski

\begin{proof}[Proof of Proposition \ref{bg3}(ii)]  We follow the method presented in \cite{rosen2005}. For the variation in $y$, write

\begin{equation} \label{try22}
\begin{split}
E&[|\aa_{\eps, t}'(y,A)-\aa_{\eps, t}'(y',A)|^n] \leq \\ & \frac{1}{(2\pi)^n} \Big|\int_{(A_1^1)^n} \int_{\RR^n} E[\prod_{k=1}^n e^{ip_k(B_{s_k}-B_{r_k})}] \prod_{k=1}^n |p_k| |e^{ip_ky}-e^{ip_ky'}| e^{-\eps p_k^2/2} dp_kdr_kds_k .
\end{split}
\end{equation}

We can write

\begin{equation} \label{dex}
\begin{split}
E[\prod_{k=1}^n e^{ip_k(B_{s_k}-B_{r_k})}] & = E[\prod_{k=1}^n e^{ip_k(B_{1/2}-B_{r_k})}\prod_{k=1}^n e^{ip_k(B_{s_k}-B_{1/2})}] \\
& = E[\prod_{k=1}^n e^{iu'_k(B_{\ell'_k}-B_{\ell'_{k-1}})}\prod_{k=1}^n e^{iu_k(B_{\ell_k}-B_{\ell_{k-1}})}],
\end{split}
\end{equation}
where $\ell'_0 < \ell'_1 < \ldots < \ell'_k=1/2$ is a relabeling of $\{r_1, \ldots , r_k, 1/2\}$ and $1/2=\ell_0 < \ell_1 < \ldots < \ell_k$ is a relabeling of $\{1/2,s_1, \ldots , s_k\}$ and each $u_k$ and $u'_k$ is the properly chosen linear combination of $p_j$'s to make \rrr{dex} valid. We may then apply local nondeterminism in the form of \rrr{van} to obtain

\begin{equation} \label{}
E[\prod_{k=1}^n e^{ip_k(B_{s_k}-B_{r_k})}] \leq e^{-C \sum_{k=1}^{n} |u_k|^2 (\ell_{j}-\ell_{j-1})^{2H}}e^{-C \sum_{k=1}^{n} |u'_k|^2 (\ell'_{j}-\ell'_{j-1})^{2H}}.
\end{equation}

We return to \rrr{try22} to obtain

\begin{equation} \label{try23}
\begin{split}
E&[|\aa_{\eps, t}'(y,A)-\aa_{\eps, t}'(y',A)|^n]  \\ & \leq C|y-y'|^{n\la} \int_{\RR^n}  \frac{\prod_{k=1}^n|p_k|^{1+\la}dp_k}{\prod_{k=1}^n (1+|u_k|^{1/H})\prod_{k=1}^n (1+|u'_k|^{1/H})} \\
 & \leq C |y-y'|^{n\la}\Big| \! \Big| \frac{\prod_{k=1}^n |p_k|^{(1+\la)/2}}{\prod_{k=1}^n (1+|u_k|^{1/H})} \Big| \! \Big|_2 \Big| \! \Big| \frac{\prod_{k=1}^n |p_k|^{(1+\la)/2}}{\prod_{k=1}^n (1+|u_k|^{1/H})} \Big| \! \Big|_2 \\
 & \leq C |y-y'|^{n\la}\Big| \! \Big| \prod_{k=1}^n\frac{1+|u_k|^{(1+\la)/2}+|u_k|^{(1+\la)}}{(1+|u_k|^{1/H})} \Big| \! \Big|_2 \Big| \! \Big| \prod_{k=1}^n \frac{ 1+|u'_k|^{(1+\la)/2}+|u'_k|^{(1+\la)}}{(1+|u_k|^{1/H})} \Big| \! \Big|_2;
\end{split}
\end{equation}
the first inequality employs the bounds \rrr{bnd76} and \rrr{clay}, the second is Cauchy-Schwarz, and the third is due to the fact that each $p$ value can be realized as $u_k-u_{k-1}$ and $u'_{k'}-u'_{k'-1}$ for some $k,k'$. The condition $\la < \frac{1}{H} - \frac{3}{2}$ implies that the integrals in the last expression of \rrr{try23} are finite, and we have therefore shown that the variation in $y$ satisfies \rrr{rundmc}. As before, the same argument with \rrr{bnd77} in place of \rrr{bnd76} handles the variation in $\eps$. For the variation in $t$, performing the steps \rrr{try100} and \rrr{try25} and then following with the method in \rrr{try23} gives us

\begin{equation} \label{try23}
\begin{split}
E&[|\aa_{\eps, t}'(y,A)-\aa_{\eps, t'}'(y,A)|^n]  \\ & \leq C|t-t'|^{n\bb} \Big| \! \Big| \prod_{k=1}^n\frac{1+|u_k|^{1/2}+|u_k|}{(1+|u_k|^{\ga/H})} \Big| \! \Big|_2 \Big| \! \Big| \prod_{k=1}^n \frac{ 1+|u'_k|^{1/2}+|u'_k|}{(1+|u_k|^{\ga/H})} \Big| \! \Big|_2,
\end{split}
\end{equation}
with $\ga = 1 - \bb$. Again the condition $\bb < 1- \frac{3}{2}H$ implies that these integrals are finite. This completes the proof of Proposition \ref{bg3} (ii)
\end{proof}

We will now discuss the possible relation of this proposition to the conjecture. To begin with, an examination of the proof of Proposition \ref{bg3}$(i)$ shows that configurations without isolated intervals can be uniformly bounded as $\kappa \lar 0$; the difficulty therefore lies with the configurations containing isolated intervals. One approach to the proof of the conjecture might be to show that the renormalization, i.e., the subtraction of the term $t \E[\dslt(y)]$, cancels with integrals over configurations with isolated intervals.

Another possibility is to use Proposition \ref{bg3}$(ii)$ and appeal to the work done by Rosen in the $H=1/2$ case. Rosen began by defining
\be A_k^j := [(2k-2)2^{-j},(2k-1)2^{-j}] \times [(2k-1)2^{-j},(2k)2^{-j}].\ee
A simple scaling and Proposition \ref{bg3}$(ii)$ show that $\dslt(y,A^j_k)$ exists and is jointly continuous in $y$ and $t$.  Note that
 $\DD = \cup_{j=1}^{\ff} \cup_{k=1}^{2^{j-1}} A_k^j$, and observe that when $H=1/2$ and $j$ is fixed,
\be\label{indsets}
 \left\{\aaa_{t,\eps}(y,A_{k}^{j})\right\}_{1\le k\le 2^{j-1}}
 \ee
  are independent. In \cite{rosen2005},  this independence was used  together with the following lemma \cite[Prop. 3.5.2]{garsia1970}, to establish $L^p$ bounds and H\"older continuity for  $\aa_{t,\eps}(y, A)$ which sufficed to
 show Kolmogorov's continuity criterion for \eqref{renormalzation}.
\begin{lemma}
Suppose $X_1, \ldots, X_n$ are independent with $E[X_j]=0$ for all $j$ and $M = \max_{1 \leq j \leq n} \E[X_j^{2p}] < \ff$, with $p$ a positive integer. Let $a_1, \ldots, a_n \in \RR$. Then
\begin{equation} \label{}
E[|a_1 X_1 + \ldots + a_n X_n|^{2p}] \leq C(p)M(a_1^2 + \ldots + a_n^2)^p.
\end{equation}
\end{lemma}

The difficulty when $1/2<H<2/3$ is that we no longer have independence in \eqref{indsets}; however, it may be that the local nondeterminism of FBM is enough. Perhaps a substitute for the above lemma can be deduced under the weaker condition of local nondeterminism, and this could be used to prove the conjecture.

\end{document}